\documentclass[11pt,a4paper]{article}

\usepackage{palatino}
\usepackage{graphicx}
\usepackage{faktor}
\usepackage{amsmath}
\usepackage{amsfonts}
\usepackage{amsthm}
\usepackage{mathrsfs}
\usepackage{latexsym}
\usepackage{amssymb}
\usepackage{amscd}
\usepackage{tikz-cd}
\usepackage[all]{xy}
\usepackage{authblk}
\usepackage{hyperref}
\usepackage{enumerate}
\usepackage{multicol}

\usepackage[scale=0.75]{geometry}

\usepackage{stackrel}
\usepackage{tensor}  
\usepackage[fleqn,tbtags]{mathtools,nccmath} 

\numberwithin{equation}{section}

\newtheorem{prop}[equation]{Proposition}
\newtheorem{lemma}[equation]{Lemma}
\newtheorem{theorem}[equation]{Theorem}
\newtheorem{corol}[equation]{Corollary}

\theoremstyle{remark}
\newtheorem{remark}[equation]{Remark}

\theoremstyle{definition}
\newtheorem{example}[equation]{Example}

\DeclareSymbolFont{AMSb}{U}{msb}{m}{n}
\DeclareMathSymbol{\N}{\mathbin}{AMSb}{"4E}
\DeclareMathSymbol{\Z}{\mathbin}{AMSb}{"5A}
\DeclareMathSymbol{\R}{\mathbin}{AMSb}{"52}
\DeclareMathSymbol{\Q}{\mathbin}{AMSb}{"51}
\DeclareMathSymbol{\I}{\mathbin}{AMSb}{"49}
\DeclareMathSymbol{\C}{\mathbin}{AMSb}{"43}

\newcommand{\dbl}{[\hspace{-0.2ex}[}
\newcommand{\dbr}{]\hspace{-0.2ex}]}
\newcommand{\db}[1]{\dbl {#1} \dbr}

\newcommand{\dblc}{(\hspace{-0.2ex}(}
\newcommand{\dbrc}{)\hspace{-0.2ex})}
\newcommand{\dbc}[1]{\dblc {#1} \dbrc}

\newcommand{\ctens}{\widehat{\otimes}}

\newcommand{\iso}{\cong}
\newcommand{\invlim}{\underleftarrow{\textnormal{lim}}\,}

\newcommand{\incl}{\hookrightarrow}
\newcommand{\id}{\textnormal{id}}

\newcommand{\Hom}{\textnormal{Hom}} 					

\newcommand{\dash}{\textnormal{-}}
\newcommand{\tn}[1]{\textnormal{#1}}
\newcommand{\cat}[1]{\tn{\textbf{#1}}} 					

\newcommand{\Cot}{\textnormal{Cot}} 					
\newcommand{\CGQ}{\textnormal{GQ}} 				
\newcommand{\wCGQ}{\widetilde{\textnormal{GQ}}} 

\newcommand{\wCK}{\widetilde{k}} 
\newcommand{\CK}{k} 

\newcommand{\Gr}{\textnormal{G}} 					
\newcommand{\Pri}[1]{\textnormal{P}_{#1}} 				
\newcommand{\qPri}[1]{\overbar{\textnormal{P}}_{#1}}

\newcommand{\overbar}[1]{\mkern 1.5mu\overline{\mkern-1.5mu#1\mkern-1.5mu}\mkern 1.5mu}

\providecommand{\keywords}[1]
{
  \small	
  \textbf{\textit{Keywords:}} #1
}

\title{A functorial approach to Gabriel $k$-quiver constructions for coalgebras and pseudocompact algebras}

\author[a]{Kostiantyn Iusenko}
\author[b]{John William MacQuarrie}
\author[a,b]{Samuel Quirino}
\affil[a]{Instituto de Matem\'{a}tica e Estat\'{i}stica, Univ. de São Paulo, São Paulo, SP, Brazil}
\affil[b]{Universidade Federal de Minas Gerais, Belo Horizonte, MG, Brazil}

\begin{document}

\footnotetext{\textit{Email addresses:} iusenko@ime.usp.br (Kostiantyn Iusenko), john@mat.ufmg.br (John MacQuarrie), quirino@ime.usp.br (Samuel Quirino)}

\maketitle

\begin{abstract}
We define the path coalgebra and Gabriel quiver constructions as functors between the category of $k$-quivers and the category of pointed $k$-coalgebras, for $k$ a field.  We define a congruence relation on the coalgebra side, show that the functors above respect this relation, and prove that the induced Gabriel $k$-quiver functor is left adjoint to the corresponding path coalgebra functor.  We dualize, obtaining adjoint pairs of functors (contravariant and covariant) for pseudocompact algebras. Using these tools we describe precisely to what extent presentations of coalgebras and algebras in terms of path objects are unique, giving an application to homogeneous algebras.
\end{abstract}


\section{Introduction}

Let $k$ be a field.  A $(k$-)coalgebra is defined in the monoidal category of $k$-vector spaces by axioms dual to those of an associative, unital $k$-algebra.  Coalgebras have been studied extensively since their introduction for at least two reasons: Firstly, they form ``half the structure'' of Hopf algebras, whose applications range from group theory to physics (we refer to \cite{Abe,DNR,Montgomery} and references therein).  And secondly, due to the fact that coalgebras have very strong finiteness properties, making them a natural context in which to generalize concepts and results from finite dimensional algebras and their representations (e.g. \cite{Green,Simson4,Simson}, and the references therein).

While the formal properties of coalgebras are very pleasant to work with, explicit calculations can be unwieldy.  For this reason, a standard trick when working with a coalgebra $C$ is to pass to its vector space dual $C^*$, which inherits naturally the structure of a topological, associative, unital $k$-algebra.  The class of algebras dual to the class of coalgebras is precisely the class of pseudocompact algebras \cite{Brumer,Simson1}, and thus understanding pseudocompact algebras and their representations provides useful tools when working with coalgebras.  But pseudocompact algebras are of independent interest, appearing for example as completed group algebras of profinite groups, so that the understanding of their structure and representations has applications in Galois theory, finite group theory, algebraic geometry and more.

\medskip

The combinatorial approach to the representation theory of finite dimensional algebras begins with two fundamental constructions: given a (pointed) finite dimensional algebra $A$ one may construct a finite directed graph, referred to as the (Gabriel) quiver of $A$.  In the other direction, beginning with a finite quiver $Q$, one may construct an associative algebra, the (complete) path algebra of $Q$.  In \cite{IM}, the first two authors of this article describe these constructions as a pair of adjoint functors, utilizing a certain intermediate category of ``Vquivers'', in which the arrows of a quiver are replaced with vector spaces (in this article, we refer to such objects as ``$k$-quivers'' rather than ``Vquivers'', the former being a more common name in the literature, e.g. \cite[Section 7]{Gabriel} and \cite[Section 5.1]{Keller}).  The adjunction presented there thus gives a very precise explanation of what information one can obtain about a finite dimensional algebra in terms of the underlying combinatorial structure.  On the other hand, the adjunction has several limitations: firstly, the category of Vquivers presented in \cite{IM} is rather unnatural, in the sense that the morphisms are not intuitive.  Secondly, only finite quivers and (essentially) only finite dimensional algebras are considered.  Thirdly, in the category of algebras, only algebra homomorphisms that are surjective modulo the radical are permitted.  
  
The Gabriel quiver construction and path algebra construction have been dualized (e.g. \cite[Section 1]{Montgomery95}, \cite[Section 8]{Simson1}, \cite{Woodcock}) and can be applied to arbitrary (pointed) coalgebras.  In this article we consider the category $k\dash\cat{Quiv}$ of $k$-quivers, and show that we have a pair of functors ``Path coalgebra'', from $k\dash\cat{Quiv}$ to the category of pointed coalgebras $\cat{PCog}$, and ``Gabriel $k$-quiver'', from $\cat{PCog}$ to $k\dash\cat{Quiv}$.  We define a natural equivalence relation $\sim$ on the morphisms of $\cat{PCog}$, observe that the functors above can be interpreted as functors between $k\dash\cat{Quiv}$ and the quotient category $\cat{PCog}_{\sim}$ and show that, interpreted this way, the Gabriel $k$-quiver functor is left adjoint to the path coalgebra functor (Theorem \ref{Theorem Adjunction Level 1}).  This adjunction improves on the main result of \cite{IM} in every way: the category of $k$-quivers presented here is far more natural; there are no finiteness assumptions; there are no hypotheses applied to coalgebra homomorphisms. 
Using certain dualities, we also give two pairs of adjoint functors where on the algebraic side we have pseudocompact algebras:  Firstly, a pair of contravariant functors adjoint on the left between $k\dash\cat{Quiv}$ and a quotient $\cat{PAlg}_{\sim}$ of the category $\cat{PAlg}$ of pointed pseudocompact algebras.  Secondly, a pair of covariant adjoint functors, which can be treated as a direct generalization of \cite[Theorem 5.2]{IM}.

The paper is organized as follows.  In Section \ref{Section Prelims} we introduce the basic definitions and concepts we require from coalgebras and comodules.  In Section \ref{Section Categories and functors} we define the categories and functors of interest.  In Section \ref{Section Coalgebra adjunction} we prove the main result (Theorem \ref{Theorem Adjunction Level 1}) of the paper, an adjunction between the functors defined in Section \ref{Section Categories and functors}.  We also give some simple examples and immediate consequences, and a brief comparison with an adjunction due to Radford.  In Section \ref{Section PC algebras and adjunctions} we dualize the theory from the previous sections, obtaining versions for pseudocompact algebras of Theorem \ref{Theorem Adjunction Level 1}. In Section \ref{Sec.uniqueness of presentations} we use  the main result to explain to what extent the presentation of a (co)algebra in terms of its path (co)algebra is unique, and as application prove that if two quotients of a completed path algebra by homogeneous closed ideals $I,L$ are isomorphic and if $I$ has degree $n$, then $L$ also has degree $n$.

\medskip

\textbf{Acknowledgements.} We would like to thank William Chin and Mark Kleiner for helpful discussions concerning this research. We are also grateful to Eduardo N.\ Marcos who drew our attention to an application of our results (see Remark \ref{EduQ}). The first author was partially
supported by FAPESP grants 2014/09310-5 and 2018/23690-6. The second author was  partially supported by CNPq grant 443387/2014-1, FAPEMIG grant PPM-00481-16 and UFMG ADRC grant Edital 05/2016. The third author was partially supported by CAPES -- Finance Code 001.

\section{Preliminaries}\label{Section Prelims}
\subsection{Coalgebras}

Fix a field $k$. Algebras, coalgebras, vector spaces, linear maps and tensor products are over $k$ unless specified otherwise. 

For a general introduction to coalgebras and comodules, see, for instance, \cite{Abe,DNR,Montgomery,Sweedler}. Given a coalgebra $C$, denote its comultiplication by $\varDelta_C$ and its counity by $\varepsilon_C$. By $\cat{Cog}$ we denote the category of all coalgebras and coalgebra  homomorphisms. If $C$ and $D$ are coalgebras and $M$ is a $C$-$D$-bicomodule, write $\mu_M:M\to C\otimes M$ for the structure of the left $C$-comodule $M$ and $\nu_M:M\to M\otimes D$ for the structure of the right $D$-comodule $M$. To simplify notation, we drop the subscript whenever there is no chance of confusion and make use of the \emph{sigma notation} (or Sweedler notation) \cite[Sections 1.2 and 2.0]{Sweedler}. In analogy with  \cite[Corollary 2.61]{Rotman} we have that the category of $C$-$D$-bicomodules is equivalent to the category of right $C^{cop}\otimes D$-comodules (in which $C^{cop}$ is the coopposite coalgebra of $C$ defined in the usual way, e.g. \cite[Definition 1.1.5]{Montgomery}). Observe that if $\rho:M\to N$ is any homomorphism of $C$-$D$-bicomodules, then $\rho$ is also a homomorphism of right $C^{cop}\otimes D$-comodules, and vice-versa.

The \emph{coradical} $C_0$ of the coalgebra $C$ is the sum of the simple left (or right) subcomodules of $C$. Thus, $C$ is \emph{cosemisimple} if, and only if, $C=C_0$. 
 
Define inductively $C_n$ to be the largest subcomodule of $C$ with the property that $C_n/C_{n-1}$ is cosemisimple.  Each $C_n$ is in fact a subcoalgebra and can be calculated as follows:
\[C_n\coloneqq \varDelta^{-1}(C\otimes C_{n-1}+C_0\otimes C).\]
The family $\{C_n\}_{n\in \N}$ is the \emph{coradical filtration} of $C$, 
where $\N$ denotes the set of all natural numbers including zero.  We have that $C = \bigcup_{n\in\N}C_n$.  Throughout this text, any numbered subscript on a coalgebra refers to its coradical filtration.  For these facts and more about the coradical filtration, see for instance \cite[Section 5.2]{Montgomery}.  We occasionally use the helpful convention $C_{-1} := \{0\}$.  We need one more useful fact:

\begin{lemma}\label{Lemma injective mod Cn}
Let $\rho:C\to D$ be an injective coalgebra homomorphism.  For each $n\in \N$ the induced map $C/C_n \to D/D_n$ is injective.
\end{lemma}

\begin{proof}
The image $\rho(C)$ is a subcoalgebra of $D$ isomorphic to $C$.  By \cite[Corollary 2.3.7]{HR74} 
$$\rho(C_n) = \rho(C)_n = \rho(C)\cap D_n.$$
It follows that if $\rho(x)\in D_n$, then $x$ is in $C_n$, as required.
\end{proof}

Given a coalgebra $C$ and $C$-bicomodule $M$, denote by $\Cot_C(M)$ the \emph{cotensor coalgebra}  $$\Cot_C(M)\coloneqq\bigoplus_{i=0}^{\infty}M^{\Box_i},$$ 
with $M^{\Box_0}\coloneqq C$ and $M^{\Box_n}\coloneqq (M^{\Box_{n-1}})\Box_C M$, wherein $\Box_C$ denotes the cotensor product over $C$ (see, \cite[Section 1.4]{Nichols} for details). 

The cotensor coalgebra is given by a universal property, which we present below.  Note that a coalgebra homomorphism $\rho:D\to C$ makes $D$ into a $C$-bicomodule with structure maps $\mu=(\rho\otimes \id)\varDelta_D$ and $\nu=(\id\otimes \rho)\varDelta_D$;  
the canonical projection  $\pi_0:\Cot_C(M)\to C$ is a coalgebra homomorphism;  
the canonical projection $\pi_1:\Cot_C(M)\to M$ is a $C$-bicomodule homomorphism.

\begin{prop}[Universal Property of the Cotensor Coalgebra, {\cite[Proposition 1.4.2]{Nichols}}]\label{cUniversalProperty}
Let $C$ and $D$ be coalgebras and $M$ a $C$-bicomodule. Given a coalgebra homomorphism $\rho_0:D\to C$, and a $C$-bicomodule homomorphism $\rho_1:D\to M$ with the property that $\rho_1$ vanishes on $D_0$, then there exists a unique coalgebra homomorphism $\rho:D\to \Cot_{C}(M)$ making the following diagrams commute
\[\begin{tikzcd}
 & \Cot_C(M)\arrow[d,"\pi_0"] \\
D \arrow[ur, dashed, "\rho"]\arrow[r, "\rho_0" near end] & C
\end{tikzcd}
\quad
\begin{tikzcd}
& \Cot_C(M)\arrow[d,"\pi_1"] \\
D \arrow[ur, dashed, "\rho"]\arrow[r, "\rho_1" near end] & M
\end{tikzcd}\]
\end{prop}

\begin{remark} 
If $C=\bigoplus_{i=0}^{\infty}C_{(i)}$ is a \emph{graded coalgebra} such that $C_n=\bigoplus_{i=0}^{n} C_{(i)}$ for every $n$, then $C$ is \emph{coradically graded} (see \cite[Chapter 2.4.1]{Abe} and \cite[Lemma 2.2]{CM}). For instance, if $C$ is a cosemisimple coalgebra and $M$ a $C$-bicomodule, then $\Cot_C(M)=\bigoplus_{i=0}^{\infty}M^{\Box_i}$ is coradically graded \cite[Lemma 4.4]{Woodcock}.
\end{remark}

\subsection{Pointed coalgebras}

A coalgebra is \emph{pointed} if every simple subcoalgebra is one dimensional.  Denote by $\cat{PCog}$ the full subcategory of $\cat{Cog}$ having objects pointed coalgebras. Denote by $\Gr(C)$ the set $\{ g\in C\,|\, \varDelta(g)=g\otimes g,\, \varepsilon(g)=1\}$ of \emph{group-like elements} of the coalgebra $C$. The elements of $\Gr(C)$ are linearly independent in $C$ \cite[Proposition 3.2.1]{Sweedler}. For any set $S$, the \emph{group-like coalgebra} on $S$, $kS$, is the vector space with basis $S$ and maps $\varDelta(s)=s\otimes s$, $\varepsilon(s)=1$, extended linearly for all $s\in S$. In particular, $k\Gr(C)$ is the \emph{group-like subcoalgebra} of $C$.

A one-dimensional subcoalgebra $D\subseteq C$ is necessarily of the form $k\{g\}$, for some $g\in \Gr(C)$ (\cite[Lemma 8.0.1]{Sweedler}). Consequently, 
\begin{remark}\label{pointedkgc}
A coalgebra $C$ is pointed if and only if $C_0=kG(C)$.
\end{remark}

Given $g,h\in \Gr(C)$, denote by $\Pri{g,h}(C)\coloneqq\{ p\in C\,|\, \varDelta(p)=p\otimes g+h\otimes p\}$ the set of all \emph{$g,h$-primitive elements}. Note that the linear maps $\mu(p)=h\otimes p$ and $\nu(p)=p\otimes g$ make $\Pri{g,h}(C)$ a $k\Gr(C)$-bicomodule.  Note also that coalgebra homomorphisms respect group-like and primitive elements.

The next results describe some structure of pointed coalgebras based on their coradical filtrations. 
\begin{prop}[{\cite[Theorem 5.4.1]{Montgomery}}]\label{pointeddescription}
Let $C$ be a pointed coalgebra. Then
\begin{enumerate}[(i)]
\item the vector space $C_1$ has a decomposition
\[C_1=k\Gr(C)\oplus\Big(\!\!\!\bigoplus_{g,h\in \Gr(C)}\!\!\!\Pri{g,h}'(C)\Big),\] 
where $\Pri{g,h}'(C)$ is any vector space complement of the vector space $\langle h-g \rangle$ in $\Pri{g,h}(C)$, i.e. $\Pri{g,h}(C)=\langle h-g \rangle\oplus \Pri{g,h}'(C)$;
\item for any $n\geqslant 1$ and $c\in C_n$,
\[c=\!\sum_{g,h\in \Gr(C)}\!\!\!c_{g,h},\text{ where } \varDelta(c_{g,h})=c_{g,h}\otimes g+h\otimes c_{g,h}+\omega_{g,h}\]
for some $\omega_{g,h}\in C_{n-1}\otimes C_{n-1}$.
\end{enumerate}
\end{prop}

Thus $\langle h-g \rangle$ is a subbicomodule of $\Pri{g,h}(C)$. Let $\qPri{g,h}(C)$ be the quotient bicomodule $\Pri{g,h}(C)/\langle h-g \rangle$ and write its elements as $\overline{p}=p+\langle h-g \rangle$.

For a pointed coalgebra $C$, Proposition \ref{pointeddescription} implies that $C_1=C_0+\sum_{g,h\in \Gr(C)}\Pri{g,h}(C)$. Hence, the structure maps of each $C_0$-bicomodule $\Pri{g,h}(C)$ induce a pair of structure maps making $C_1$ a $C_0$-bicomodule. Moreover,
\[\faktor{C_1}{C_0}=\!\sum_{g,h\in\Gr(C)}\!\!\!\frac{\Pri{g,h}(C)+C_0}{C_0}\iso\!\bigoplus_{g,h\in\Gr(C)}\!\!\!\qPri{g,h}(C).\]

\begin{prop}\label{Coradical filtration and homomorphisms}
Let $C$ and $D$ be coalgebras with $C$ pointed and  $\rho:C\to D$ a coalgebra map. Then $\rho(C_n)\subseteq D_n$ for all $n\in \mathbb N$.
\end{prop}
\begin{proof}
Follows from \cite[Theorem 9.1.4]{Sweedler}.
\end{proof}

\subsection{Quivers and path coalgebras}

A \emph{quiver} $Q=(Q_0,Q_1,s,t)$ is a \emph{directed graph}, i.e.\ a set of vertices $Q_0$, a set of arrows $Q_1$, and two functions $s,t:Q_1\rightrightarrows Q_0$, where for any arrow $\alpha\in Q_1$, $s(\alpha)$ represents its source and $t(\alpha)$ represents its target \cite[Section III.1]{ARS}. 
A \emph{map of quivers} $\phi:Q\to R$ consists of a function $\phi_0:Q_0\to R_0$ together with a function $\phi_1:Q_1\to R_1$ such that $\phi(s(a)) = s(\phi(a))$ and  $\phi(t(a)) = t(\phi(a))$ for every $a\in Q_1$. 
Denote by $\cat{Quiv}$ the category of quivers and maps of quivers.

A \emph{path} in $Q$ of \emph{length} $l\geqslant 1$ is the formal composition of arrows $a_la_{l-1}\dots a_1$ with $s(a_j)=t(a_{j-1})$.  To each vertex $i\in Q_0$ we associate a \emph{stationary path} $e_i$ of length $|e_i|=0$ with $s(e_i)=t(e_i)=i$.

The \emph{path coalgebra} $\CK Q$ of the quiver $Q$ is the vector space with basis all paths in $Q$, with comultiplication and counity maps given by
\begin{align*}
\varDelta(w)=&\sum_{w=w_2w_1}\!\!\!w_2\otimes w_1, \qquad \varepsilon(w)=\delta_{|w|0}.
\end{align*}
In this way $\CK Q\cong \Cot_{kQ_0}(\tn{span}\{Q_1\})$ \cite[Section 4]{Woodcock}. Hence, $\CK Q$ is pointed, $\Gr(\CK Q)$ consists of the stationary paths, $(\CK Q)_0=kQ_0$, and $\CK Q$ is coradically graded with coradical filtration $\{(\CK Q)_{\leqslant m}\}_{m\in \N}$, where $(\CK Q)_{\leqslant m}$ is generated as a vector space by all paths of $Q$ of length $m$ or less. 

Given a pointed coalgebra $C$, one constructs the Gabriel quiver 
of $C$ as follows: the set of vertices is the set $\Gr(C)$ and the set of arrows from $g$ to $h$ is a basis of the quotient space $\qPri{g,h}(C)$ \cite[Description 4.12]{Simson}. 
The choice of these bases means that the construction is not functorial.

\section{Categories and functors}\label{Section Categories and functors}
\subsection{Category of $k$-quivers}\label{subsection VQuiv def}

A \emph{$k$-quiver} $VQ = (VQ_0, VQ_{g,h})$ consists of a set of vertices $VQ_0$ together with a $k$-vector space $VQ_{g,h}$ for each (ordered) pair $g,h\in VQ_0$. A \emph{map of $k$-quivers} $\varphi=(\varphi_0,\varphi_{g,h}):(VQ_0,VQ_{g,h}) \to (VR_0,VR_{g',h'})$ 
consists of 
\begin{itemize}
\item a function $\varphi_0:VQ_0\to VR_0$.
\item a linear map $\varphi_{g,h}:VQ_{g,h}\to VR_{\varphi_0(g),\varphi_0(h)}$ for each pair of vertices $g,h\in VQ_0$.
\end{itemize}
The category $k\dash\cat{Quiv}$ has objects $k$-quivers and morphisms maps of $k$-quivers.  One might compare this definition with the more awkward \cite[Definitions 3.1 and 3.2]{IM}.

There exists a correspondence between quivers and $k$-quivers: given a quiver  $Q=(Q_0,Q_1)$, for each pair of vertices $g,h\in Q_0$, the vector spaces $Q_{g,h} \coloneqq \langle a\in Q_1\,|\,s(a)=g,t(a)=h \rangle $ define a $k$-quiver $VQ=(Q_0,Q_{g,h})$; on the other hand, if we start with a $k$-quiver $VQ=(VQ_0,VQ_{g,h})$, we obtain a quiver by taking as arrows from $g$ to $h$ a basis of $VQ_{g,h}$. The first correspondence (with the obvious assignment for morphisms) defines a functor, which we denote by $V(-):\cat{Quiv}\to k\dash\cat{Quiv}$. The second correspondence does not.  We observe in passing that the functor $V(-)$ of course does possess a forgetful right adjoint, but we make no use of this functor here.

\begin{example}

\[Q: 
\begin{tikzcd}
 1 \arrow[rr,bend left=15,"\alpha"] \arrow[rr,bend right=15,"\beta"'] && 2
 \arrow[out=30,in=-30,loop,"\gamma"] &, & 
 VQ=V(Q):\ \ 1 \arrow[rr,"{\langle\alpha,\beta\rangle}"]  && 2
 \arrow[out=30,in=-30,loop,"{\langle\gamma\rangle}"]
\end{tikzcd}
\]
In this example, the $k$-quiver $VQ$ has vertices $1,2$ and arrow spaces given by
\[VQ_{i,j}= 
\begin{cases}
\langle \alpha,\beta \rangle \iso k^2 & \hbox{if } i=1,\,j=2\\
\langle \gamma \rangle \iso k &\hbox{if }  i=2,\,j=2\\
\{0\} & \textnormal{otherwise}
\end{cases}\]
\end{example}

One of the main advantages of the relationship between quivers and coalgebras is that one obtains a combinatorial description of the comodules for a given coalgebra in terms of representations of quivers.  We mention that working with $k$-quivers we maintain this advantage.  Representations of $k$-quivers are defined and their relation to (co)modules discussed, for instance, in \cite[Section 7]{Gabriel} and \cite[Section 5]{Simson07}.

\subsection{``Close'' coalgebra homomorphisms}

Given two coalgebra homomorphisms $\rho, \gamma: C\to D$, write $\rho\sim\gamma$ if 
\begin{align*}
&(\rho-\gamma)(C_0)=0, \quad \mbox{and}\\
&(\rho-\gamma)(C_1)\subseteq D_{0}.
\end{align*}
It is easy to check (cf.\ \cite[Section 3.2]{IM}) that $\sim$ is a congruence relation on $\cat{PCog}$. 
By $\cat{PCog}_{\sim}$ we denote the corresponding quotient category. 

\begin{prop}[cf.\ {\cite[Proposition 4]{TW}}]\label{tiln}
Let $\rho,\gamma:C\to D$ be two homomorphisms in $\cat{PCog}$ such that $\rho\sim\gamma$. Then $(\rho-\gamma)(C_i)\subseteq {D_{i-1}}$, for each $i\geqslant 0$.
\end{prop}
\begin{proof}
We proceed by induction on $i$. Suppose that $(\rho-\gamma)(C_i)\subseteq {D_{i-1}}$ for every $i\leqslant n-1$. Observe that \[\varDelta_D(\rho-\gamma)=(\rho\otimes\rho-\gamma\otimes\gamma)\varDelta_C=(\rho\otimes(\rho-\gamma)+(\rho-\gamma)\otimes\gamma)\varDelta_C\] 
since $\rho$ and $\gamma$ are coalgebra homomorphisms. Also, \cite[Corollary 9.1.7]{Sweedler} shows that $\varDelta_C(C_n)\subseteq \sum_{i=0}^{n} C_i\otimes C_{n-i}$. Thus (applying Proposition \ref{Coradical filtration and homomorphisms}) we get
\begin{align*}
\varDelta_D(\rho-\gamma)(C_{n})&=  (\rho\otimes(\rho-\gamma)+(\rho-\gamma)\otimes\gamma)\varDelta_C(C_{n})\\
&\subseteq  (\rho\otimes(\rho-\gamma)+(\rho-\gamma)\otimes\gamma)\Bigl(\sum_{i=0}^{n} C_i\otimes C_{n-i}\Bigr)\\
&\subseteq  \sum_{i=0}^n D_i\otimes D_{n-1-i}+\sum_{i=0}^n D_{i-1}\otimes D_{n-i}\\
 & =  \sum_{i=0}^{n-1} D_i\otimes D_{n-1-i}\subseteq D\otimes D_{n-2}+D_0\otimes D.
\end{align*}
Hence $(\rho-\gamma)(C_{n})\subseteq D_{n-1}$.
\end{proof}

Working in the quotient category $\cat{PCog}_{\sim}$ rather than $\cat{PCog}$, much of the important information is preserved.  For instance:

\begin{prop} 
The projection functor $\Pi:\cat{PCog}\to \cat{PCog}_{\sim}$ reflects isomorphisms.  That is, if $\rho:C\to D$ is a coalgebra homomorphism such that $\Pi(\rho): C\to D$ is an isomorphism, then $\rho$ is an isomorphism.
\end{prop}

\begin{proof}
It is sufficient to show that for any coalgebra endomorphism $\rho:C\to C$, $\rho\sim\id_C$ implies that $\rho$ is an isomorphism. Let $\rho: C\to C$ be a coalgebra homomorphism such that $\rho\sim \id$. Since $C=\bigcup_{n\geqslant 0}C_n$, any element $c\in C$ belongs to $C_n$ for some $n\in\N$. 

Suppose that $c\in \ker(\rho)$ is not $0$, so that $c\in C_n\!\!\setminus\! C_{n-1}$ for some $n\geqslant 0$. By Proposition \ref{tiln}, \[(\id-\rho)(c)=c-\rho(c)=c\in C_{n-1},\]
contradicting our hypothesis. Hence $c=0$ and, consequently, $\rho$ is injective. 

Let $c_0=c\in C_n\!\!\setminus\! C_{n-1}$ and define recursively $c_i=\rho(c_{i-1})-c_{i-1}\in C_{n-i}$, for $i=1,\dots,n$. This sequence stops at $\rho(c_n)-c_n=0$. Writing $c'=\sum_{i=0}^{n}(-1)^{i}c_i$ we get $\rho(c')=c$. Thus $\rho$ is surjective and this completes the proof.
\end{proof}

\begin{prop}
If $\rho:C\to D$ is an injective map in $\cat{PCog}$ then its image in $\cat{PCog}_{\sim}$ is a monomorphism. 
\end{prop}
\begin{proof}
 Suppose $\gamma,\sigma: B\to C$ are two coalgebra homomorphisms such that $\rho\circ\gamma\sim \rho\circ\sigma$.
For any $b\in B_0$ we have
\[(\rho\circ\gamma-\rho\circ\sigma)(b)=\rho(\gamma(b)-\sigma(b))=0\Longleftrightarrow \gamma(b)-\sigma(b)=0,\]
since $\rho$ is injective. For $b'\in B_1$, we have
\[(\rho\circ\gamma-\rho\circ\sigma)(b')=\rho(\gamma(b')-\sigma(b'))\subseteq D_0\stackrel{\ref{Lemma injective mod Cn}}{\Longleftrightarrow} \gamma(b')-\sigma(b')\subseteq C_0.\]
Thus $\gamma\sim\sigma$ and the result follows.
\end{proof}

\subsection{Path coalgebra and Gabriel $k$-quiver functors}\label{subsection Cofunctors definitions}

We define functors between the categories introduced above.

Given a $k$-quiver $VQ=(VQ_0,VQ_{g,h})$, denote by $\Sigma_Q=(kVQ_0,\varDelta_0, \varepsilon_0)$ the group-like coalgebra of $VQ_0$, and by $V_Q=(VQ_1,\mu,\nu)$ the $\Sigma_Q$-bicomodule $VQ_1=\bigoplus_{g,h\in VQ_0} VQ_{g,h}$ with structure maps:
\[\mu (m_{g,h})= h\otimes m_{g,h}, \qquad
\nu (m_{g,h})= m_{g,h}\otimes g,\]
for each $m_{g,h}\in VQ_{g,h}$.

Define the \emph{path coalgebra} $\CK[VQ]$ as the cotensor coalgebra $\Cot_{\Sigma_Q}(V_Q)$.
For any $\varphi=(\varphi_0,\varphi_{g,h})$ in $\Hom_{k\dash\cat{Quiv}}(VQ,VR)$, the universal property of the cotensor coalgebra, Proposition \ref{cUniversalProperty}, ensures the existence of a unique homomorphism $\rho\in\Hom_{\cat{PCog}}(\CK[VQ],\CK[VR])$ making the following diagrams commutative: 
\[\begin{tikzcd}
\Cot_{\Sigma_Q}(V_Q) \arrow[d,"\pi_0'"'] \arrow[r, dashed, "\rho"] \arrow[dr, "\rho_0"] & \Cot_{\Sigma_R}(V_R)\arrow[d,"\pi_0"] \\
\Sigma_Q \arrow[r, "\varphi_0"] & \Sigma_R
\end{tikzcd}
\quad
\begin{tikzcd}
\Cot_{\Sigma_Q}(V_Q) \arrow[d,"\pi_1'"'] \arrow[r, dashed, "\rho"] \arrow[dr, "\rho_1"] & \Cot_{\Sigma_R}(V_R)\arrow[d,"\pi_1"] \\
V_Q \arrow[r, "\varphi_1"] & V_R
\end{tikzcd}\]
where $\pi_i',\pi_i$ are the canonical projections, $\varphi_i$ are linear extensions of the maps defined by $\varphi$, and $\rho_i\coloneqq \varphi_i\circ\pi_i'$, for $i=0,1$. Set $\CK[\varphi]\coloneqq\rho$.

\begin{example}
If $\iota:VQ\incl VR$ is an inclusion of $k$-quivers, then $\CK[\iota]:\CK[VQ]\to \CK[VR]$ is the corresponding inclusion of coalgebras.
\end{example}

These constructions yield a covariant functor $\CK[-]:k\dash\cat{Quiv}\to \cat{PCog}$. Denote by $\wCK[-]:k\dash\cat{Quiv}\rightarrow \cat{PCog}_{\sim}$ the covariant functor $\Pi\circ \CK[-]$. 

\medskip

Let $C$ be a pointed coalgebra. Define the \textit{Gabriel $k$-quiver} of $C$ by 
$$\CGQ(C)\coloneqq(\CGQ(C)_0,\CGQ(C)_{g,h}),$$ where $\CGQ(C)_0\coloneqq\Gr(C)$ and for each pair of vertices $g,h\in \CGQ(C)_0$, the vector space $\CGQ(C)_{g,h}$ is defined to be $\qPri{g,h}(C)$ (see after Proposition \ref{pointeddescription}).

Let $\rho\in \Hom_{\cat{PCog}}(C,D)$. Observe that, by the isomorphism theorems for comodules, there exists a unique comodule homomorphism $\bar{\rho}:\faktor{C\!}{\!C_0}\to \faktor{D\!}{\!D_0}$ such that the following diagram is commutative:
\[\begin{tikzcd}
C\rar{\rho}\ar[d,"{\pi^{C}}"'] & D\dar{\pi^D}\\
\faktor{C\!}{\!C_0}\rar{\bar{\rho}} & \faktor{D\!}{\!D_0}
\end{tikzcd}\]
The maps
\[\varphi_0\coloneqq \left.\rho\right|_{\Gr(C)}:\Gr(C)\to \Gr(D),\quad\varphi_{g,h}\coloneqq \left.{\bar{\rho}}\right|_{\qPri{g,h}(C)}:\qPri{g,h}(C)\to\qPri{\varphi_0(g),\varphi_0(h)}(D),\]
define a map of $k$-quivers $\varphi=(\varphi_0,\varphi_{g,h}):\CGQ(C)\to \CGQ(D)$.  This construction yields a covariant functor
$\CGQ(-):\cat{PCog}\to k\dash\cat{Quiv}$. Furthermore,

\begin{prop}
There is a unique functor  
$\wCGQ(-):\cat{PCog}_{\sim} \to k\dash\cat{Quiv}$ such that $\CGQ(-)=\wCGQ(-)\circ \Pi$.
\end{prop}

\begin{proof}
Using Remark \ref{pointedkgc} and Proposition \ref{pointeddescription}, one checks that defining $\wCGQ(C)$ to be $\CGQ(C)$ and $\wCGQ([\rho])$ to be $\CGQ(\rho)$, we obtain a covariant functor satisfying the claim. It is clearly unique.
\end{proof}

\begin{example}\label{example triangle}
A simple example of a path coalgebra is given by the $k$-quiver
\[VQ=\begin{tikzcd}
 \circ_1 \arrow[rr,"{\langle a \rangle}"] \arrow[dr, "{\langle b \rangle}"'] & & \circ_3\\
 & \circ_2 \ar[ur, "{\langle c \rangle}"']
\end{tikzcd}.\]
The coalgebra $\CK[VQ]$ is a 7 dimensional vector space with basis $\{e_1,e_2,e_3,a,b,c,cb\}$, where $e_i$ are group-like elements. The comultiplication of $cb$, for example, is given by
$$\varDelta(cb)=e_3\otimes cb+c\otimes b+cb\otimes e_1.$$ 
Let $\rho:\CK[VQ]\to \CK[VQ]$ be the linear map that sends $a$ to $a+(e_3-e_1)$ and fixes all other elements of the given basis. Then $\rho$ is a coalgebra automorphism and $\CGQ(\rho)=\id_{VQ}$.  Thus $\CGQ(-)$ is not faithful.
\end{example}

\section{Adjunction and consequences}\label{Section Coalgebra adjunction}
\subsection{The main result and its proof}

We prove that the functor $\wCK[-]$ is right adjoint to $\wCGQ(-)$. To do this, we show that the counit $\varepsilon : \wCGQ(\wCK[-])\to \tn{id}_{k\dash\cat{Quiv}}$ and unit $\eta : \tn{id}_{\cat{PCog}_{\sim}} \to \wCK[\wCGQ(-)]$ are given as follows:
\begin{itemize}
\item Given $VQ\in k\dash\cat{Quiv}$,
$$\varepsilon_{VQ} : \wCGQ(\wCK[VQ])\to VQ$$
is the $k$-quiver map sending $e_i\in \tn{G}(\wCK[VQ])$ to $e_i \in VQ$ and for any $e,f \in \tn{G}(\wCK[VQ])$ the element $x+\langle e-f \rangle\in \qPri{e,f}(\wCK[VQ])$ is sent  to $x\in VQ_{e,f}$.  The maps $\varepsilon_{VQ}$ are easily checked to be the components of a natural transformation $\varepsilon: \wCGQ(\wCK[-])\to \tn{id}_{k\dash\cat{Quiv}}$;
\item Given $C\in \cat{PCog}$, choose a coalgebra splitting $s:C\to C_0$ of the inclusion $i_0:C_0\to C$ (which exists because the coradical is separable \cite[Section 2.3.4]{Abe}). We treat $C$ as a $C_0$-bicomodule via $s$ and choose a splitting of the inclusion of bicomodules $i_1:C_1\to C$ (which exists because $C_1$ is an injective comodule \cite[Theorem 3.1.5]{DNR}).
Combining this splitting with the natural projection map $C_1 \to C_1/C_0$ we get a map $t:C\to C_1/C_0$. 
The maps $s,t$ define (by the universal property of the cotensor coalgebra) the map $\eta_C^{s,t}:C\to \Cot_{C_0}\Bigl(\faktor{C_1}{C_0}\Bigr)=\wCK[\wCGQ(C)]$. 
\end{itemize}

The congruence class of $\eta_C^{s,t}$ in $\cat{PCog}_{\sim}$ does not depend on the choice of splittings $s,t$, so we may denote $\eta_C^{s,t}$ simply by $\eta_C$. Indeed suppose that $s,t$ and $s',t'$ are two different choices, and $\eta_C^{s,t}$, $\eta_C^{s',t'}$ are the corresponding maps. We must confirm that $\eta_C^{s,t}\sim\eta_C^{s',t'}$. One has 
\begin{align*}
    (\eta_C^{s,t}-\eta_C^{s',t'})|_{C_0}=\pi_0(\eta_C^{s,t}-\eta_C^{s',t'})i_0=si_0-s'i_0=0
\end{align*}
and the relation $(\eta_C^{s,t}-\eta_C^{s',t'})(C_1)\subseteq \CK[\CGQ(C)]_0$ may be checked in a similar manner.

\begin{remark}
The map $\eta_C$ is the image in $\cat{PCog}_{\sim}$ of the coalgebra embedding considered in \cite[Corollary 1]{Rad82} and \cite[(4.8)]{Woodcock} (cf.\ also \cite[Theorem 4.3]{CMo} and \cite[Theorem 3.1]{CZ06}).
\end{remark}

\begin{lemma}
The map $\eta_C:C\to \wCK[\wCGQ(C)]$ is the component at $C$ of a natural transformation $\eta : \tn{id}_{\cat{PCog}_{\sim}} \to \wCK[\wCGQ(-)]$.
\end{lemma}

\begin{proof}
Let $\rho: C\to D$ be a morphism in $\cat{PCog}$. We must check that the following square commutes in $\cat{PCog}_{\sim}$:
\[\begin{tikzcd}
C \ar[rr,"\rho"] \ar[d,swap,"{\eta_C}"] && D\ar[d,"{\eta_D}"] \\
\wCK[\wCGQ(C)] \ar[rr,swap,"{\wCK[\wCGQ(\rho)]}"] && \wCK[\wCGQ(D)] \\
\end{tikzcd}
\]
As above choose maps $s,t$ which split inclusions $i_0^C:C_0\to C$ and $C_1\to C$ respectively, and $s', t'$ which split  inclusions $i_0^D:D_0\to D$ and $D_1 \to D$ respectively. Denote by $\tilde \rho$ the map $\wCK[\wCGQ(\rho)]$. 
We have that 
\begin{align*}
    (\eta_D^{s',t'}\rho-\tilde \rho \eta_C^{s,t})|_{C_0}&=\pi_0^D(\eta_D^{s',t'}\rho-\tilde \rho \eta_C^{s,t})i_0^C\\
    &=(s'\rho\, i_0^C-\rho|_{C_0}\pi_0^C \eta_C^{s,t})i_0^C\\
    &=s' i_0^D\rho|_{C_0}-\rho|_{C_0} s\, i_0^C \\
    &=\rho|_{C_0}-\rho|_{C_0} \\
    &=0.
\end{align*}
One similarly confirms that $(\eta_D^{s',t'}\rho-\tilde \rho \eta_C^{s,t})(C_1)\subseteq D_0$. Hence the classes of $\eta_D^{s',t'}\rho$ and $\tilde \rho \eta_C^{s,t}$ are equal in $\cat{PCog}_{\sim}$ and $\eta$ is a natural transformation.
\end{proof}
 
\begin{theorem}\label{Theorem Adjunction Level 1}
The functor $\wCK[-]:k\dash\cat{Quiv} \to \cat{PCog}_{\sim}$ is right adjoint to the functor $\wCGQ(-):\cat{PCog}_{\sim} \to k\dash\cat{Quiv}$.
\end{theorem}

\begin{proof}
We check that the counit-unit equations hold. That is, that for any $C\in \cat{PCog}_{\sim}$,
$$
    \tn{id}_{\wCGQ(C)}=\varepsilon_{\wCGQ(C)}\circ \wCGQ(\eta_c)
$$
and that for any $VQ\in k\dash\cat{Quiv}$,
$$
    \tn{id}_{\wCK[VQ]}=\wCK[\varepsilon_{VQ}]\circ \eta_{\wCK[VQ]}.
$$
The first equality is a straightforward verification using the definitions of the unit and counit. The second equality translates as $\CK[\varepsilon_{VQ}]\circ \eta^{s,t}_{\CK[VQ]} \sim \tn{id}_{\CK[VQ]}$,
where $s,t$ are two splittings as in the construction of the unit $\eta$ and $\eta^{s,t}_{\CK[VQ]}$ is the corresponding morphism
$$
   \eta^{s,t}_{\CK[VQ]}:  \CK[VQ] \to  \CK[\CGQ(\CK[VQ])].
$$
One checks that
\[\begin{tikzcd} \CK[\CGQ(\CK[VQ])]_0\ar[dr, "{\widetilde{i_0}}"] \\ \Sigma_Q\ar[r,"{i_0}"']\ar[u,"{(\varepsilon_{VQ})_0^{-1}}"] & \CK[VQ]\rar{\eta^{s,t}_{\CK[VQ]}}\ar[rr, bend right=20, "s"] & \CK[\CGQ(\CK[VQ])]\rar{\pi'_0} & \CK[\CGQ(\CK[VQ])]_0\rar{(\varepsilon_{VQ})_0} & \Sigma_Q \end{tikzcd}\]
commutes and hence the composition of the horizontal maps $\Sigma_Q \to \Sigma_Q$ is the identity map (because $s$ is a splitting of $\widetilde{i_0}$). Therefore,
\begin{align*}
    (\CK[\varepsilon_{VQ}]\circ \eta^{s,t}_{\CK[VQ]} - \tn{id}_{\CK[VQ]})|_{\Sigma_Q}&=\pi_0^{\CK[VQ]}(\CK[\varepsilon_{VQ}]\circ \eta^{s,t}_{\CK[VQ]} - \tn{id}_{\CK[VQ]})i_0\\
    &=(\varepsilon_{VQ})_0\pi'_0 \eta^{s,t}_{\CK[VQ]}i_0 - \pi_0^{\CK[VQ]} i_0\\
    &=\tn{id}_{\Sigma_Q}-\tn{id}_{\Sigma_Q}=0.
\end{align*}
Similarly we get $(\CK[\varepsilon_{VQ}]\circ \eta^{s,t}_{\CK[VQ]} - \tn{id}_{\CK[VQ]})(\CK[VQ]_1)\subseteq \Sigma_Q$ and the second equation follows.
\end{proof}

\subsection{Consequences and examples}

\begin{remark}
Recall that a subcoalgebra $H$ of a path coalgebra $k[VQ]$ is \emph{admissible} if $H$ contains  $\CK[VQ]_1$ \cite[Definition 4.7]{Woodcock}.
If $C$ is a pointed coalgebra, any representative in $\cat{PCog}$ of the unit map $\eta_C : C\to \wCK[\wCGQ(C)]$ of Adjunction \ref{Theorem Adjunction Level 1} realizes $C$ as an admissible subcoalgebra of its path coalgebra.  For further discussion about the uniqueness of such a presentation of $C$, see Section \ref{Sec.uniqueness of presentations}.
\end{remark}

\begin{remark}
Using Lemma \ref{Lemma injective mod Cn}
and the Heyneman-Radford Theorem (see e.g.\ {\cite[Theorem 5.3.1]{Montgomery}}) one shows that the Adjunction \ref{Theorem Adjunction Level 1} restricts to an adjunction between the wide subcategories of $k\dash\cat{Quiv}$ and $\cat{PCog}_{\sim}$ with morphisms the monomorphisms.
\end{remark}

\begin{remark} A coalgebra $C$ is said to be \textit{hereditary}
\cite{NTZ} if homomorphic images of injective comodules are injective.  It is known (e.g.\ \cite[Theorem 1]{Chin}) that $C$ is hereditary if, and only if, $C$ is isomorphic to $\wCK[\wCGQ(C)]$.  Therefore, if we restrict $\cat{PCog}_{\sim}$ to the full subcategory of hereditary coalgebras, the Adjunction \ref{Theorem Adjunction Level 1} yields an adjoint equivalence of categories. 
\end{remark}

\begin{remark} 
Each component of the unit is a monomorphism and each component of the counit is an isomorphism.
It follows by abstract nonsense (cf.\ \cite[Theorem IV.3.1]{MacLane}) that the functor $\wCGQ(-)$ is faithful and that $\wCK[-]$ is fully faithful.
\end{remark}

\begin{remark} \label{RemPhi}
The unit and counit of  Adjunction \ref{Theorem Adjunction Level 1} define bijections
\[\Psi=\Psi_{C,VQ}:\Hom_{\cat{PCog}_{\sim}}(C,\wCK[VQ])\to\Hom_{k\dash\cat{Quiv}}(\wCGQ(C),VQ),\]
with $\Psi([\rho])=\varepsilon_{VQ}\wCGQ([\rho])$ and $\Psi^{-1}(\varphi)=\wCK[\varphi]\eta_C$ (cf.\ \cite[Theorem IV.1.2]{MacLane}).
\end{remark}

\begin{remark}
Adjunction \ref{Theorem Adjunction Level 1} may be compared with a similar, but different adjunction due to Radford \cite{Rad82}.  On the ``combinatorial side'', Radford's category $(\mathscr{SV})_k$ is equivalent to $k\dash\cat{Quiv}$, but the ``algebraic'' categories $(\mathscr{C}_p\mathscr{E})_k$ and $\cat{PCog}_{\sim}$ are non-equivalent.  While the left adjoint functor $\wCGQ(-)$ above corresponds to the Gabriel $k$-quiver construction, the left adjoint functor in \cite{Rad82} is better thought of as giving a Peirce decomposition of a coalgebra (cf.\ \cite[Section 2.1]{HGK10} for Peirce decompositions of algebras or \cite{CG02} for a related approach to coalgebras using idempotents).  In order to see that the functors are fundamentally different, one may observe that the image of the unit map of Radford's adjunction applied to the coalgebra $k[VQ]$ of Example \ref{example triangle} does not yield an admissible subcoalgebra.  In particular, the results of Section \ref{Sec.uniqueness of presentations} do not follow formally from the adjunction in \cite{Rad82}.
\end{remark}

\begin{example}
The adjunction above allows us to describe the automorphisms of the path coalgebra $\wCK[VQ]$ in terms of automorphisms of the corresponding $k$-quiver $VQ$.
In the following examples we suppress notation: an arrow that should be labelled with a vector space of dimension $1$ 
    will be left unlabelled. 
\begin{enumerate}
    \item Consider the following $k$-quivers:
    \[A_{\infty}: \begin{tikzcd}
    \circ \arrow[r] & \circ \arrow[r] & \circ \arrow[r] & \cdots
    \end{tikzcd}\]
    \[\prescript{}{\infty}{A}_{\infty}: \begin{tikzcd}
    \cdots \arrow[r] & \circ \arrow[r]  & \circ \arrow[r] & \cdots
    \end{tikzcd}\]
    An automorphism of $\wCK[A_{\infty}]$ must fix the vertices.  Indeed, $$\tn{Aut}_{\cat{PCog}_{\sim}}(\wCK[A_{\infty}]) \iso \prod_{n\in \N}k^{\times}$$
    where $k^{\times}$ is the group of units of $k$ and the product is indexed by the arrow spaces.
    
    An automorphism of $\wCK[\prescript{}{\infty}A_{\infty}]$ can shift the vertices.  Indeed  $\tn{Aut}_{\cat{PCog}_{\sim}}(\wCK[\prescript{}{\infty}A_{\infty}])$ is isomorphic to a semidirect product 
    $$\left(\prod_{n\in \Z}k^{\times}\right) \rtimes \Z.$$
    Note that the automorphism groups of both these algebras in $\cat{PCog}$ are quite a bit larger, because for example in $\cat{PCog}_{\sim}$ we don't distinguish between the identity and the automorphism that sends the element $x$ of the arrow space $e\to f$ to $x + (f-e)$.

    \item If $VQ$ is the $k$-quiver with one vertex and a loop indexed by the vector space $V$, then we have $\tn{Aut}_{\cat{PCog}_{\sim}}(\wCK[VQ])\iso\tn{GL}(V) = \tn{Aut}_k(V)$.  The $k$-quivers of this form are the only connected $k$-quivers for which the corresponding automorphism groups in $\cat{PCog}$ and in $\cat{PCog}_{\sim}$ are equal.
    
    \item For the \emph{Kronecker} $k$-quiver \[K_V: \begin{tikzcd}
    \circ \rar{V} & \circ 
    \end{tikzcd}\]
    with $V$ a $k$-vector space, we also have that \[\tn{Aut}_{\cat{PCog}_{\sim}}(\wCK[K_V])\iso \tn{GL}(V).\]

\end{enumerate}
\end{example}

\section{Pseudocompact Algebras}\label{Section PC algebras and adjunctions}

\subsection{Preliminaries and categories}

Throughout this section $k$ remains a field, now treated as a discrete topological ring.  A \emph{pseudocompact algebra} is an associative, unital, Hausdorff topological $k$-algebra $A$ possessing a basis of neighborhoods of 0 consisting of (open) ideals $I$ having cofinite dimension in $A$ that intersect in 0 and such that $A \iso \invlim_I A/I$. Denote by $\cat{Alg}$ the category of pseudocompact algebras and continuous homomorphisms

Let $A,B$ be pseudocompact algebras.  A \emph{pseudocompact $A$-$B$-bimodule} is a topological $A$-$B$-bimodule $U$ possessing a basis of 0 consisting of open subbimodules $V$ of finite codimension that intersect in 0 and such that $U \iso \invlim_V U/V$. By $J(A)$ denote the Jacobson radical of $A$; i.e. the intersection of the maximal closed left ideals of $A$ (see \cite[Section 1, p.444]{Brumer} for alternative characterizations of $J(A)$).

We must be a little careful when defining the higher radicals of $A$.  Given a pseudocompact $A$-module $M$, define $\tn{Rad}(M)$ to be the intersection of the maximal closed $A$-submodules of $M$.  For $n\geqslant 1$, we define $J^{n+1}(A) = \tn{Rad}(J^n(A))$.  There seems no reason to suppose that the abstract submodule of $A$ generated by $J(A)\cdot J(A)$ be closed in $A$, but we have that
$$J^2(A) = \overline{J(A)\cdot J(A)}.$$
This can be seen taking limits, observing that $\frac{J(A)J(A) + I}{I} = J(A/I)J(A/I)$ for every open ideal $I$ of $A$ and using that the equality $J(B)\cdot J(B) = \tn{Rad}(J(B))$ holds for finite dimensional $B$. 

Recall the duality between pseudocompact algebras and coalgebras, formalized by Simson in \cite{Simson}.  Given a topological vector space $V$, let $V^* = \Hom_k(V,k)$ denote the set of continuous functionals on $V$.  If $C$ is a coalgebra (always treated as discrete), then $C^*$ inherits naturally the structure of a pseudocompact algebra (the ``dual algebra of $C$''), while if $A$ is a pseudocompact algebra, then $A^*$ inherits naturally the structure of a coalgebra (the ``dual coalgebra of $A$'').
In this way we obtain a duality of categories (see \cite[Theorem 3.6]{Simson})
\begin{equation*} 
\begin{tikzcd}
\cat{Cog}\ar[r, shift left, "(-)^*"]& \cat{Alg}.\ar[l,shift left, "(-)^{*}"] 
\end{tikzcd}
\end{equation*}
Similarly, the functors $(-)^*$ induce a duality between the category of pseudocompact $A$-modules and the category of comodules over $A^{*}$ (or, equivalently, $C$-comodules and pseudocompact $C^*$-modules), see \cite[Theorem 4.3]{Simson1} for details.

A pseudocompact algebra $A$ is \emph{pointed} if every quotient of $A$ by a closed maximal left ideal is one dimensional, or equivalently if $A/J(A)$ is isomorphic as a topological algebra to a product of copies of $k$.   Denote by $\cat{PAlg}$ the full subcategory of $\cat{Alg}$ consisting of all pointed pseudocompact algebras.   By a (topologically) semisimple pseudocompact algebra $A$ we mean an algebra such that $J(A)=0$.  This condition is equivalent to saying that $A$ is isomorphic to a direct product of simple finite dimensional algebras (properly interpreted, the proof of \cite[Theorem 16]{Kaplansky47} goes through for pseudocompact algebras.  Alternatively, the result is a special case of \cite[Theorem 2.10]{Iovanov2}). 
The duality $(-)^*$ between coalgebras and pseudocompact algebras restricts to a duality between the full subcategories of cosemisimple pointed coalgebras and semisimple pointed pseudocompact algebras, respectively.

Given a  pseudocompact algebra\! $A$ and a pseudocompact $A$-bimodule $U$\!, denote by $T\db{A,U}$ the \textit{complete tensor algebra} 
$T\db{A,U} \coloneqq \prod_{n=0}^{\infty}U^{\ctens_n}$,
with $U^{\ctens_0} \coloneqq A$ and $U^{\ctens_n}\coloneqq (U^{\ctens_{n-1}}) \ctens_A U$, where $\ctens_A$ denotes the complete tensor product over $A$ (see \cite[Section 7.5]{Gabriel} for details).  The universal property for the complete tensor algebra is given in \cite[Lemma 2.11]{IM}.  One may check that if $M$ is a left $C$-comodule and $W$ is a right $C$-comodule, then 
$$W\Box_C M\cong (W^*\ctens_{C^*} M^*)^*.$$
Hence the pseudocompact algebra dual to $\tn{Cot}_C(M)$ is $T\db{C^*,M^*}$.

The \emph{complete path algebra} $k\db{Q}$ of the quiver $Q$ is the set of sequences $(\lambda_w)_w$ indexed by (oriented) paths in $Q$, with multiplication defined by \[(\lambda_w)_w*(\kappa_v)_v=\Bigl(\sum_{u=wv}\lambda_w\kappa_v\Bigr)_u\]
(see \cite[Section 1]{Iovanov}). It follows that $k\db{Q}$ is a pseudocompact algebra.  It is a standard fact that $(\CK Q)^{*}$ is isomorphic to the complete path algebra $k\db{Q}$ (e.g. \cite[Proposition 8.1]{Simson1}).

\medskip

Let $\alpha, \beta: A\to B$  be two homomorphisms in $\cat{PAlg}$. We write $\alpha\sim \beta$  if
\begin{align*}
&(\alpha-\beta)(A)\subseteq J(B), \quad \mbox{and}\\
&(\alpha-\beta)(J(A))\subseteq J^2(B).
\end{align*}
As with coalgebras, one easily checks that $\sim$ defines a congruence relation on $\cat{PAlg}$.  We denote by $\cat{PAlg}_{\sim}$ the corresponding quotient category.  The relation $\sim$ for pseudocompact algebras is dual to the relation $\sim$ for coalgebras in the following sense:

\begin{prop}\label{sim_dual}
Let $\rho,\gamma:C\to D$ be two homomorphisms in $\cat{PCog}$. Then $\rho\sim\gamma$ if, and only if, $\rho^*\sim\gamma^*$ in $\cat{PAlg}$. 
\end{prop}

\begin{proof}
If $\rho', \gamma':A\to B$ are homomorphisms of pseudocompact algebras, the condition $\rho'\sim \gamma'$ can be interpreted as saying that the compositions
$$A\xrightarrow{\rho'-\gamma'}B \to B/J(B)\,,\quad J(A)\xrightarrow{\rho'-\gamma'}J(B) \to J(B)/J^2(B)$$
are the zero map, while if $\rho, \gamma:C\to D$ are homomorphisms of coalgebras, the condition $\rho\sim \gamma$ can be interpreted as saying that the compositions
$$C_0\to C \xrightarrow{\rho-\gamma} D\,,\quad C_1/C_0\to C/C_0 \xrightarrow{\rho-\gamma} D/D_0$$
are the zero map.  The proposition is thus a formal consequence of duality.
\end{proof}

\begin{prop}\label{prop duality with tilde}
The duality functors $(-)^*$ between $\cat{PCog}$ and $\cat{PAlg}$ induce a duality between the categories $\cat{PCog}_{\sim}$ and $\cat{PAlg}_{\sim}$
\end{prop}

\begin{proof}
Immediate from Proposition \ref{sim_dual}. 
\end{proof}

One proves as in \cite[Lemma 3.8]{IM} (or by dualizing Proposition \ref{tiln}) that given $\alpha,\beta:A\to B$ in $\cat{PAlg}$, if $\alpha\sim\beta$ then $(\alpha - \beta)(J^n(A))\subseteq J^{n+1}(B)$ for every $n\geqslant 0$.

\subsection{Contravariant adjoint functors}

We obtain a new, contravariant adjunction immediately from the adjunction of Theorem \ref{Theorem Adjunction Level 1} and the duality of categories of Proposition \ref{prop duality with tilde}:

Define the contravariant functors
\begin{align*}
\widetilde{k}\db{-} : k\dash\cat{Quiv}  & \to \cat{PAlg}_{\sim} \\
               VQ\,\, & \mapsto \wCK[VQ]^{*}
\end{align*}
and 
\begin{align*}
\widetilde{\tn{GQ}}\dbc{-} : \cat{PAlg}_{\sim}  & \to k\dash\cat{Quiv} \\
               A\,\,\,\,\, & \mapsto \widetilde{\tn{GQ}}(A^{*}).
\end{align*}
with the obvious definition for morphisms.  Recall that the pair of contravariant functors $F : \mathcal{C} \leftrightarrow \mathcal{D} : G$ are \emph{adjoint on the left} if for each pair of objects $c\in \mathcal{C}$ and $d\in\mathcal{D}$ we have a natural isomorphism
$$\tn{Hom}_{\mathcal{D}}(Fc , d) \to \tn{Hom}_{\mathcal{C}}(Gd, c).$$
We have

\begin{theorem}\label{theorem pc left adjunction}
The functors $\widetilde{\tn{GQ}}\dbc{-}, \widetilde{k}\db{-}$ are adjoint on the left.
\end{theorem}

\begin{proof}
This is completely formal.  Given $A\in \cat{PAlg}$ and $VQ\in k\dash\cat{Quiv}$ we have
\begin{align*}
    \tn{Hom}_{k\dash\cat{Quiv}}(\widetilde{\tn{GQ}}\dbc{A} , VQ) & = \tn{Hom}_{k\dash\cat{Quiv}}(\widetilde{\tn{GQ}}(A^{*}) , VQ) \\
    & \iso \tn{Hom}_{\cat{PCog}_{\sim}}(A^* , \wCK[VQ]) \\
    & \iso \tn{Hom}_{\cat{PAlg}_{\sim}}(\wCK[VQ]^{*} , A^{**}) \\
    & \iso \tn{Hom}_{\cat{PAlg}_{\sim}}(\widetilde{k}\db{VQ} , A),
\end{align*}
as required.
\end{proof}

\subsection{Covariant adjoint functors}

In \cite{IM}, the first two authors of this article define a pair of covariant adjoint functors between a certain category of finite $k$-quivers and a category whose objects are pseudocompact pointed algebras $A$ such that $A/J^2(A)$ is finite dimensional and whose morphisms are (congruence classes of) those algebra homomorphisms $\alpha:A\to B$ such that the induced map $A/J(A) \to B/J(B)$ is surjective.  The adjunctions \ref{Theorem Adjunction Level 1} and \ref{theorem pc left adjunction} 
are far more general, because there are no finiteness assumptions and there are no conditions on the algebra homomorphisms.
We show in this section that if one is willing to leave behind the notion of quiver, one can in fact extend the adjunction of covariant functors \cite[Theorem 5.2]{IM} to this same level of generality.

\medskip

The category $k\dash\cat{Quiv}$ defined in Section $\ref{subsection VQuiv def}$ is equivalent to the ``category of pairs'' $\cat{ParCog}$, whose definition is as follows: objects are pairs $(\Sigma, V)$, where $\Sigma$ is a pointed cosemisimple coalgebra and $V$ is a $\Sigma$-bicomodule.  A morphism
$$(\Sigma, V)\to (\Sigma', V')$$
is a pair $(\varphi_0, \varphi_1)$ consisting of a coalgebra homomorphism $\varphi_0 : \Sigma\to \Sigma'$ and a $\Sigma'$-bicomodule homomorphism $\varphi_1:V\to V'$, with $V$ treated as a $\Sigma'$-bicomodule via $\varphi_0$. The functor $k\dash\cat{Quiv}\to \cat{ParCog}$ sends the $k$-quiver $VQ$ to the pair $(\Sigma_Q,V_Q)$, where
$$\Sigma_Q = \bigoplus_{g\in VQ_0}k\,,\quad V_Q = \bigoplus_{g,h\in VQ_0}VQ_{g,h}.$$
The bicomodule structure is as in Section \ref{subsection Cofunctors definitions}.  The action on morphisms is obvious.  In the other direction, we define the functor $\cat{ParCog}\to k\dash\cat{Quiv}$ by sending $(\Sigma, V)$ to the $k$-quiver $VQ$ having vertices $VQ_0 = \textnormal{G}(\Sigma)$ and for each pair $g,h\in \textnormal{G}(\Sigma)$, 
$$VQ_{g,h} = \{v\in V\,|\, \mu(v) = h\otimes v\hbox{ and } \nu(v) = v\otimes g\}.$$
The action on morphisms is again obvious.  Observing that $V = \bigoplus VQ_{g,h}$ because $\textnormal{G}(\Sigma)$ is a basis for $\Sigma$, one checks that these functors give the affirmed equivalence of categories.

Dually, define the category $\cat{ParAlg}$ to be the category whose objects are pairs $(A,U)$ with $A$ a pointed topologically semisimple pseudocompact algebra and $U$ a pseudocompact $A$-bimodule.  A morphism $(A,U)\to (A',U')$ is a pair $(\varphi_0,\varphi_1)$ consisting of a continuous algebra homomorphism $\varphi_0:A\to A'$ and a continuous $A$-bimodule homomorphism $\varphi_1: U\to U'$, with $U'$ treated as an $A$-bimodule via $\varphi_0$.  The categories $\cat{ParCog}$ and $\cat{ParAlg}$ are clearly dual via $(\Sigma,V)\mapsto (\Sigma^*, V^*)$.  By composing, the category $k\dash\cat{Quiv}$ is dual to the category $\cat{ParAlg}$.

One could alternatively dualize the category of $k$-quivers directly, but this is awkward and one loses combinatorial intuition anyway, because the dual of a map of (normal) $k$-quivers that is not injective on vertices will not be a map of directed graphs between the dual quivers (vertices do not go to vertices).

Consider the covariant functor
$$T\db{-} : \cat{ParAlg}\to \cat{PAlg}_{\sim}$$
given on objects by $T\db{(A,U)}:= T\db{A,U}$ and on morphisms via the universal property of the complete tensor algebra, and also the covariant functor
$$G\db{-} : \cat{PAlg}_{\sim} \to \cat{ParAlg}$$
given on objects by $A\mapsto (A/J(A) , J(A)/J^2(A))$ and on morphisms in the obvious way.  We have the following diagram of categories and functors, wherein arrows marked $E$ are equivalences and arrows marked $D$ are dualities:
\[\begin{tikzcd}
\cat{ParCog}\ar[r,leftrightarrow,"E"]\ar[d,leftrightarrow,swap,"D"]      & k\dash\cat{Quiv}\ar[r,bend left=8,"{\wCK[-]}"] &   \cat{PCog}_{\sim} \ar[l,bend left=8,"{\widetilde{\tn{GQ}}(-)}"]
\ar[d,leftrightarrow,"D"]   \\ 
\cat{ParAlg}\ar[rr,bend left=6,"{T\db{-}}"]       &             &   \cat{PAlg}_{\sim}\ar[ll,bend left=6,"{G\db{-}}"] 
\end{tikzcd}\]
\begin{prop}\label{prop nat isos}
In the above diagram, the composition 
$$\cat{ParAlg} \to \cat{ParCog} \to k\dash\cat{Quiv} \to \cat{PCog}_{\sim} \to \cat{PAlg}_{\sim}$$
is naturally isomorphic to $T\db{-}$, and the composition
$$\cat{PAlg}_{\sim} \to \cat{PCog}_{\sim} \to k\dash\cat{Quiv} \to \cat{ParCog} \to \cat{ParAlg}$$
is naturally isomorphic to $G\db{-}$.
\end{prop}

\begin{proof}
Simple checks.
\end{proof}

\begin{theorem}\label{theorem PC covariant adjunction}
The functor $T\db{-}$ is left adjoint to the functor $G\db{-}$.
\end{theorem}

\begin{proof}
Immediate from Proposition \ref{prop nat isos} and the Adjunction \ref{Theorem Adjunction Level 1}.
\end{proof}

The main adjunction from \cite{IM} can be interpreted as a special case of Theorem \ref{theorem PC covariant adjunction}: The subcategory $\mathcal{F}$ of $\cat{ParAlg}$ whose objects are those pairs $(A,U)$ with both $A,U$ finite dimensional and whose morphisms are those $(\varphi_0,\varphi_1)$ with $\varphi_0$ surjective, is equivalent to the category of finite pointed quivers given in \cite{IM}.  On the algebra side we restrict $\cat{PAlg}_{\sim}$ to the category $\mathcal{A}$ whose objects are those algebras $A$ in $\cat{PAlg}_{\sim}$ with $A/J^2(A)$ finite dimensional, and whose morphisms are (congruence classes of) those algebra homomorphisms $A\to B$ such that the induced map $A/J(A)\to B/J(B)$ is surjective.  The functors above restrict to adjoint functors
\[\begin{tikzcd}
\mathcal{F} \ar[rr,bend left=13,"{T\db{-}}"] && \mathcal{A} \ar[ll,bend left=13,"{G\db{-}}"] 
\end{tikzcd}\]
and this adjunction is \cite[Theorem 5.2]{IM}.

\section{Uniqueness of presentations} \label{Sec.uniqueness of presentations}

We use formal properties of the functors discussed above and the definition of $\sim$ to describe precisely to what extent the presentation of a (co)algebra in terms of a path (co)algebra is unique.  
Say that an injective coalgebra homomorphism $\rho:H\to \CK[VQ]$ in $\cat{PCog}$ is \emph{admissible} if its image is an admissible subcoalgebra of $\CK[VQ]$ (that is, if $\rho(H)$ contains $\CK[VQ]_1$).  Recall from Remark \ref{RemPhi} that given a pointed coalgebra $C$ and a $k$-quiver $VQ$, we denote by $\Psi_{C,VQ}$ the hom-set isomorphism induced by the Adjunction \ref{Theorem Adjunction Level 1}.

\begin{lemma}\label{Lemma Psi of inj is iso}
Let $VQ$ be a $k$-quiver, $H$ a pointed coalgebra  and $\rho:H\to \CK[VQ]$ an admissible coalgebra homomorphism. 
Then the corresponding morphism $\Psi_{H,VQ}([\rho])$ is an isomorphism. 
\end{lemma}

\begin{proof} 
The map $\rho$ being admissible implies that the restriction $\rho:H_1\to \CK[VQ]_1$ is an isomorphism. Indeed, it is surjective because given $y\in \CK[VQ]_1$ there is $x \in H$ such that $\rho(x)=y$. But the induced map $H/{H_1}\to \CK[VQ]/{\CK[VQ]_1}$ is injective by Lemma \ref{Lemma injective mod Cn} and so $x\in H_1$. Now the result follows by construction (see Remark \ref{RemPhi}), since $\varepsilon_{VQ}$ is an isomorphism.
\end{proof}

Let $\gamma, \delta : C \to k[\tn{GQ}(C)]$ be two presentations of the coalgebra $C\in \cat{PCog}$ as an admissable subcoalgebra of its path coalgebra.

\begin{prop}\label{prop move presentations mod tilde}
There is an automorphism $\rho$ of $k[\tn{GQ(C)}]$ for which the diagram
\[
\begin{tikzcd}  & C\ar[dl,swap,"\gamma"]\ar[dr,"{\delta}"] & \\
k[\tn{GQ(C)}]\ar[rr,swap,"\rho"] && k[\tn{GQ(C)}]
\end{tikzcd}
\]
commutes in $\cat{PCog}_{\sim}$.
\end{prop}

\begin{proof}
Denote by 
$$\Psi = \Psi_{C,\tn{GQ}(C)}:\tn{Hom}_{\cat{PCog}_{\sim}}(C, k[\tn{GQ}(C)]) \to \tn{Hom}_{k\dash\cat{Quiv}}(\tn{GQ}(C), \tn{GQ}(C))$$
the adjunction isomorphism.  General properties of adjoint functors tell us that
$$k[\Psi(\gamma)]\eta_C = \gamma \,\,\hbox{ and }\,\, k[\Psi(\delta)]\eta_C = \delta.$$
By Lemma \ref{Lemma Psi of inj is iso}, $\Psi(\gamma), \Psi(\delta)$ are isomorphisms, so we obtain the automorphism
$$\Psi(\delta)\Psi(\gamma)^{-1}$$
of $\tn{GQ}(C)$.  Applying $k[-]$ to this map we obtain the automorphism $\rho = k[\Psi(\delta)\Psi(\gamma)^{-1}]$ of $k[\tn{GQ}(C)]$ and we claim that $\rho\gamma = \delta$:
\begin{align*}
    \rho\gamma & = k[\Psi(\delta)]k[\Psi(\gamma)]^{-1}\gamma \\
    & = k[\Psi(\delta)]\eta_C \\
    & = \delta,
\end{align*}
as required.
\end{proof}

\begin{prop}\label{prop deal with tilde presentations}
Given two presentations $\gamma, \delta : C \to k[\tn{GQ(C)}]$ with $\gamma \sim \delta$, there exists an automorphism $\psi$ of $k[\tn{GQ(C)}]$ with $\psi \sim \tn{id}$ and such that $\psi\gamma = \delta$.
\end{prop}

\begin{proof}
It is easier to prove the dual version of the result, which states that given two presentations of $A\in \cat{PAlg}$ 
$$\gamma'\sim \delta' : k\db{VQ}\to A,$$
there is a continuous automorphism $\psi'\sim \tn{id}$ of $k\db{VQ}$ such that $\gamma'\psi' = \delta'$.  The proof of \cite[Proposition 6.1]{IM} carries through for pseudocompact algebras, using a version for pseudocompact algebras of the Malcev Uniqueness Theorem due to Eckstein \cite[Theorem 17]{Eckstein}.
\end{proof}

Putting these together we obtain a description of the uniqueness of a presentation of a coalgebra as an admissible subcoalgebra of its path coalgebra.

\begin{corol}
Let $C,D$ be admissible subcoalgebras of the path coalgebra $k[VQ]$.  Then $C$ is isomorphic to $D$ if, and only if, there is a coalgebra automorphism of $k[VQ]$ mapping $C$ isomorphically onto $D$.
\end{corol}

\begin{proof}
If the automorphism of $k[VQ]$ exists, then $D$ is clearly isomorphic to $C$.  If $C$ is isomorphic to $D$ then apply Propositions \ref{prop move presentations mod tilde} and \ref{prop deal with tilde presentations} to obtain the required automorphism of $k[VQ]$.
\end{proof}

Recall that a \emph{relation ideal} of $k\db{VQ}$ is a closed ideal contained inside $J^2(k\db{VQ})$ (this definition corresponds by duality to ``admissible subcoalgebra'', but for algebras the term ``admissible'' is usually reserved for ideals $I$ of the form $J^n\subseteq I\subseteq J^2$ for some $n$).  Follows the dual version for pseudocompact algebras of the above corollary:

\begin{prop}\label{prop uniqueness of PC presentations}
Let $VQ$ be a $k$-quiver and $I,L$ relation ideals of $k\db{VQ}$.  Then the pseudocompact algebras $k\db{VQ}/I$ and $k\db{VQ}/L$ are isomorphic if, and only if, there is a continuous algebra automorphism of $k\db{VQ}$ sending $I$ isomorphically onto $L$.
\end{prop}

Say that a pseudocompact algebra $A$ is \emph{graded} if it can be expressed as a product of closed subspaces $A = \prod_{i\in \N}A_i$ in such a way that $a_ia_j\in A_{i+j}$ whenever $a_i\in A_i, a_j\in A_j$.  In this case, we say that $A_n$ is the \emph{degree $n$ homogeneous part} of $A$.  Observe that the pseudocompact path algebra $k\db{VQ}$ is graded, with degree $n$ homogeneous part generated as a pseudocompact vector space by the paths of length $n$.  Say that a relation ideal $I$ of $k\db{VQ}$ is \emph{homogeneous} if it is generated as a closed ideal by a set of homogeneous elements.  It has \emph{degree $n$} if it is generated by homogeneous elements of degree exactly $n$.  We present an application of Proposition \ref{prop uniqueness of PC presentations} that appears to be unknown even for finite dimensional algebras and when $n=2$. 

\begin{theorem}\label{theorem homogeneous presentations}
Let $VQ$ be a $k$-quiver and let $I,L$ be homogeneous ideals of $k\db{VQ}$ with $k\db{VQ}/I\!\iso k\db{VQ}/L$.  If $I$ has degree $n$, then so does $L$.
\end{theorem}

\begin{proof}
Write $J^n = J^n(k\db{VQ})$.  We first claim that $I\cap J^{n+1}\subseteq \tn{Rad}(I)$, where $\tn{Rad}(I)$ denotes the radical of $I$ as a $k\db{VQ}$-bimodule.  We have that
$$I\cdot J + J\cdot I \subseteq \tn{Rad}(I).$$
The ideal $I\cap J^{n+1}$ is generated as a closed ideal by elements of the form $p\cdot \alpha\in I\cdot J$ or $\alpha\cdot p\in J\cdot I$, where $p$ is a generator of $I$ of degree $n$ and $\alpha$ is an arrow of $VQ$, hence the claim.  By Proposition \ref{prop uniqueness of PC presentations}, there is a continuous automorphism $\varphi$ of $k\db{VQ}$ such that $\varphi(I) = L$ and so
$$I\cap J^{n+1}\subseteq \tn{Rad}(I) \implies \varphi(I\cap J^{n+1})\subseteq \varphi(\tn{Rad}(I)) \implies L\cap J^{n+1}\subseteq \tn{Rad}(L).$$
Suppose that $L$ is generated by a set of homogeneous elements $X$.  If ever $x\in X\cap J^{n+1}$, then $x\in \tn{Rad}(L)$ and is thus redundant.  It follows that $L$ is generated by its intersection with the degree $n$ part of $k\db{VQ}$, as required.
\end{proof}

\begin{remark} \label{EduQ}
This question was brought to our attention by Eduardo Marcos, who asked: if a finite dimensional pointed algebra $A$ has a quadratic presentation $kQ/I$ (that is, if $I$ is generated by linear combinations of paths of length $2$) and if $kQ/L$ is another homogeneous presentation of $A$, then must this presentation also be quadratic?  The positive answer to this question is a special case of Theorem \ref{theorem homogeneous presentations}.
\end{remark}

\bibliography{mref}
\bibliographystyle{alpha}

\end{document}